\documentclass[a4paper,10pt]{amsart}

\usepackage{amsmath}
\usepackage{amssymb}
\usepackage{amsthm}
\usepackage{color}
\usepackage{graphicx}
\usepackage{tikz}
\usepackage{tikz-cd}

\theoremstyle{definition}
\newtheorem{theorem}{Theorem}[section]
\newtheorem{definition}[theorem]{Definition}
\newtheorem{lemma}[theorem]{Lemma}

\newcommand\G{{\mathcal{G}}}

\DeclareMathOperator{\Aut}{Aut}
\DeclareMathOperator{\st}{st}
\newcommand\F{{\mathbb{F}}}
\newcommand\N{{\mathbb{N}}}
\newcommand\Apol{{\mathcal{A}}}

\title{Portrait growth in contracting, regular branch groups}

\author{Zoran \v{S}uni\'c}
\address{Department of Mathematics \\ Hofstra University \\ Hempstead, NY 11549, USA}
\email{zoran.sunic@hofstra.edu} 

\author{Jone Uria-Albizuri}
\address{Department of Mathematics\\ University of the Basque Country, UPV/EHU\\ Leioa, Bizkaia, Spain.}
\email{jone.uria@ehu.eus}

\date{}
\thanks{J.\ Uria-Albizuri acknowledge financial support from the Spanish Government, grants MTM2011-28229-C02 and
MTM2014-53810-C2-2-P, and from the Basque Government, grants IT753-13, IT974-16 and the predoctoral grant PRE-2014-1-347.}


\begin{document}

\begin{abstract}
We address a question of Grigorchuk by providing both a system of recursive formulas and an asymptotic result for the portrait growth of the first Grigorchuk group. The results are obtained through analysis of some features of the branching subgroup structure of the group. More generally, we provide recursive formulas for the portrait growth of any finitely generated, contracting, regular branch group, based on the coset decomposition of the groups that are higher in the branching subgroup structure in terms of the lower subgroups. Using the same general approach we fully describe the portrait growth for all non-symmetric GGS-groups and for the Apollonian group. 
\end{abstract}

\maketitle


\section{Introduction}

Regular branch groups acting on rooted trees have been extensively studied since the 1980s. The interest in these groups is due to their remarkable properties. For instance, some of the groups in this class provide counterexamples to the General Burnside Problem, many of them are amenable but not elementary amenable, and the first example of a group of intermediate word growth was of this type. The initial examples, constructed in the early 1980s, were the Grigorchuk 2-groups~\cite{grigorchuk:burnside,grigorchuk:gdegrees} and the Gupta-Sidki $p$-groups~\cite{gupta-sidki:burnside}. Many other examples and different generalizations have been introduced since then. 

The notion of word growth for a finitely generated group was introduced by A.S.~{\v S}varc~\cite{svarc:volume} and later, independently, by J.~Milnor~\cite{milnor:note,milnor:solvable}. Given a finitely generated group, one can define the word metric with respect to a finite generating set. It is natural to ask what is the word growth function of a particular group, that is, what is the number of elements that can be written as words up to a given length (equivalently, what is the volume of the ball of a given radius in the word metric). A lot of work has been done in this direction. For instance, it is known that groups of polynomial growth are exactly the virtually nilpotent groups~\cite{gromov:pol}. Since the free group of rank at least 2 has exponential growth, the word growth of every finitely generated group is at most exponential. In 1968, Milnor~\cite{milnor:problem} asked if every finitely generated group has either polynomial or exponential growth, and Grigorchuk~\cite{grigorchuk:milnor,grigorchuk:gdegrees} showed in the early 1980s that groups of intermediate growth exist. In particular, the growth of the first Grigorchuk group~\cite{grigorchuk:burnside}, introduced in 1980 as an example of a Burnside 2-group, is superpolynomial, but subexponential. 

The first Grigorchuk group is a self-similar, contracting, regular branch group, and so are the Gupta-Sidki $p$-groups. The study of such groups often relies on the representation of their elements through portraits. For instance, the known estimate~\cite{bartholdi:upper} of word growth for the first Grigorchuk group is based on an  estimate of the number of portraits of elements of given length. Thus, in the context of self-similar, contracting groups acting on regular rooted trees, it is of interest to study the portrait growth of a group with respect to a finite generating set. The necessary definitions are given in Section 2, but let us provide here a rough description. Given a finitely generated group acting on a rooted tree, one can consider the action of the group on the subtrees below each vertex of the tree. If the action on such subtrees is always given by elements of the group itself, the group is called self-similar. Moreover, if the elements describing the action at the subtrees on each level are getting shorter (with respect to the word metric) with the level, with finitely many exceptions, the group is said to be contracting. Thus, in a finitely generated, contracting group $G$, there is a finite set $\mathcal M$ such that, for every element $g$ in $G$, there exists a level such that all elements describing the action of $g$ on the subtrees below that level come from the finite set $\mathcal M$. The first level where this happens is called the depth of the element $g$ (with respect to $\mathcal M$), a concept introduced by Sidki~\cite{sidki:aut3}. In this setting, the portrait growth function (growth sequence) counts the number of elements up to a given depth and one may ask what is the growth rate of such a function. This question was specifically raised by Grigorchuk~\cite{grigorchuk:unsolved} for the first Grigorchuk group. 

We provide a system of recursive formulas for the portrait growth of the first Grigorchuk group (see Theorem~\ref{t:g-rec}) and the following asymptotic result. 

\begin{theorem}
There exists a positive constant $\gamma$  such that the portrait growth sequence $\{a_n\}_{n=0}^\infty$ of the first Grigorchuk group $\G$ satisfies the inequalities
\[
 \frac{1}{4} e^{\gamma{2^n}} \leq a_n \leq 4 e^{\gamma{2^n}},
\]
for all $n \geq 0$. Moreover, $\gamma\approx 0.71$. 
\end{theorem}

We note here that the recursive formulas from Theorem~\ref{t:g-rec}, together with the results of Lemma~\ref{l:criterion}, may be used to estimate $\gamma$ with any desired degree of accuracy. 

More generally, we provide recursive formulas (see~\eqref{eq:rec}) for the portrait growth sequence of any finitely generated, contracting, regular branch group. The formulas are directly based on the branching subgroup structure of the group.

As a further application, we also study the portrait growth for all non-symmetric GGS-groups and for the Appollonian group and, in each case, we resolve the obtained recursion and provide a straightforward description of the portrait growth. 

\begin{theorem}
Let $G$ be a GGS-group defined by a non-symmetric vector $e\in \F_p^{p-1}$. The portrait growth sequence of $G$ $\{a_n\}_{n=0}^{\infty}$ is given by

\begin{align*}
a_0&=1+2(p-1)\\
a_n&=p(x_1+(p-1)y_1)^{p^{n-1}},
\end{align*}
where $x_1$ and $y_1$ are the number of solutions in $\F_p$ of 
\[
 ((n_0, \dots, n_{p-1})C(0,e)) \odot (n_0, \dots, n_{p-1})=(0,\dots,0),
\]
with $n_0+\dots+n_{p-1}= 0$ and $n_0+\dots+n_{p-1}=1$, respectively.
\end{theorem}

For instance, for the so called Gupta-Sidki 3-group, which corresponds to the GGS-group defined by $e=(1,-1)$ with $p=3$, the portrait growth sequence is given by $a_0 = 5$ and $a_n=3\cdot 9^{3^{n-1}}$, for $n \geq 1$. 

\begin{theorem}
The portrait growth sequence $\{a_n\}_{n=0}^{\infty}$ of the Apollonian group is given by
\begin{align*}
a_{n}=3^{\frac{3^n-1}{2}} \cdot 7^{3^n} = \frac{1}{\sqrt{3}}\left( 7 \sqrt{3} \right)^{3^n}.
\end{align*}
\end{theorem}

Note that, in all three cases, the portrait growth is doubly exponential even though the word growth function is intermediate for the first Grigorchuk group, mostly unknown for the GGS-groups, and exponential for the Apollonian group. We conjecture that the portrait growth is doubly exponential for all finitely generated, contracting, regular branch groups. 

The paper is organized as follows: in Section 2 we provide the basic definitions regarding groups acting on regular rooted trees and we describe a procedure yielding recursive relations for the portrait growth sequence of a finitely generated, contracting, regular branch group. In Section 3 we state some useful observations about sequences of doubly exponential growth. Finally in Sections 4, 5 and 6 we describe the cases of the first Grigorchuk group, non-symmetric GGS-groups, and the Apollonian group. 


\section{Portrait growth sequence on a regular branch contracting group}

A regular rooted tree $T$ is a graph whose vertices are the words over a finite alphabet $X$, and two vertices $u$ and $v$ are joined by an edge if $v=ux$ for some $x\in X$. The empty word, denoted $\emptyset$, represents the root of the tree and the tree is called $d$-adic if $X$ consists of $d$ elements. The vertices represented by words of a fixed length constitute a level, that is, the words of length $n$ constitute the $n$th level of the tree. The ternary rooted tree based on $X=\{1,2,3\}$ is displayed in Figure~\ref{f:tree}. 

\begin{figure}[!ht]
\begin{center}
\begin{tikzpicture}[level distance=20mm]
\tikzstyle{level 1}=[sibling distance=40mm]
\tikzstyle{level 2}=[sibling distance=18mm]
\tikzstyle{level 3}=[sibling distance=7mm]
\coordinate (root)
child {child child child}
child {child child child}
child {child child child};

\draw[fill] (root) circle (2pt);
\node[left=5pt](root){$\emptyset$};

\draw[fill] (root-1) circle (2pt);
\node[left=4pt] at (root-1) {1};
\draw[fill] (root-2) circle (2pt);
\node[left=4pt] at (root-2) {2};
\draw[fill] (root-3) circle (2pt);
\node[left=4pt] at (root-3) {3};

\draw[fill] (root-1-1) circle (2pt);
\node[right=4pt] at (root-1-1) {11};
\draw[fill] (root-1-2) circle (2pt);
\node[left=4pt] at (root-1-2) {12};
\node[below] at (root-1-2) {$\vdots$};
\draw[fill] (root-1-3) circle (2pt);
\node[left=4pt] at (root-1-3) {13};
\draw[fill] (root-2-1) circle (2pt);
\node[right=4pt] at (root-2-1) {21};
\draw[fill] (root-2-2) circle (2pt);
\node[left=4pt] at (root-2-2) {22};
\node[below] at (root-2-2) {$\vdots$};
\draw[fill] (root-2-3) circle (2pt);
\node[left=4pt] at (root-2-3) {23};
\draw[fill] (root-3-1) circle (2pt);
\node[right=4pt] at (root-3-1) {31};
\draw[fill] (root-3-2) circle (2pt);
\node[left=4pt] at (root-3-2) {32};
\node[below] at (root-3-2) {$\vdots$};
\draw[fill] (root-3-3) circle (2pt);
\node[left=4pt] at (root-3-3) {33};

\end{tikzpicture}
\caption{A ternary rooted tree}
\label{f:tree}
\end{center} 
\end{figure}
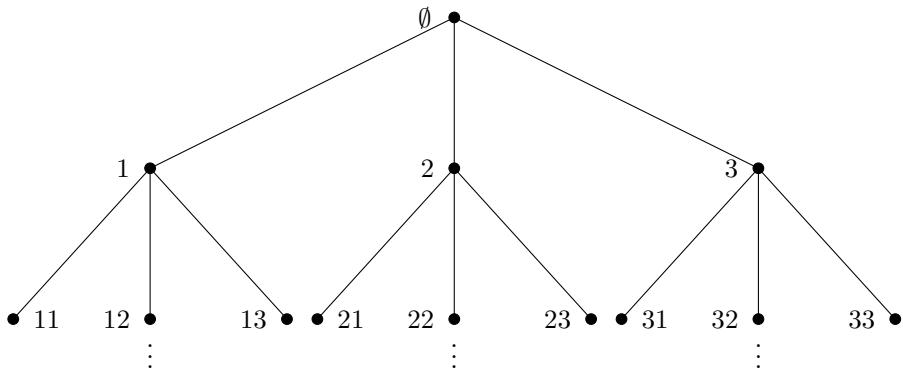

An automorphism of the tree $T$ is a permutation of the vertices preserving incidence. Such a permutation, by necessity, must also preserve the root and all other levels of the tree. The automorphisms of $T$ form the automorphism group $\Aut T$ under composition. A group acting faithfully on a regular rooted tree may be regarded as a subgroup of $\Aut T$. Every automorphism $g$ can be fully described by indicating how $g$ permutes the $d$ vertices at level 1 below the root, and how it acts on the subtrees hanging from each vertex at level 1. Observe that each subtree hanging from a vertex is isomorphic to the whole tree $T$, so that the following description makes sense. Namely, we decompose any automorphism $g\in\Aut T$ as
\begin{equation}\label{decompose}
g=(g_1,\dots,g_d)\alpha,
\end{equation}
where $\alpha\in S_d$ is the permutation describing the action of $g$ on the $d$ vertices at level 1, and each $g_i\in\Aut T$ describes the action of $g$ on the $i$th subtree below level 1, for $i=1,\dots,d$. This process can be repeated at any level, and the element describing the action of $g$ at a particular vertex $u$ is called the \textbf{section} of $g$ at the vertex $u$.
Given a group $G\leq \Aut T$ we say that $G$ is \textbf{self-similar} if the sections of each element in $G$ belong again to the group $G$. 

We say that a self-similar group $G\leq \Aut T$ is \textbf{contracting} if there is a finite set $\mathcal M$ of elements in $G$ such that, for every element $g$, there exists some level $n$ such that all sections of $g$ at the vertices at and below the $n$th level belong to $\mathcal M$. The smallest set $\mathcal M$ among such finite sets is called the \textbf{nucleus} of $G$ and we denote it by $\mathcal{N}(G)$ (or by $\mathcal N$, when $G$ is clear from the context). The definition of a contracting group, in this form, is due to Nekrashevych~\cite{nekrashevych:ss-book}. 

For any finitely generating group $G$, generated by a finite symmetric set $S=S^{-1}$, we define the word length of $g$, denoted $\partial(g)$, to be the length of the shortest word over the alphabet $S$ representing the element $g$. A self-similar group $G$ generated by a finite symmetric set $S$ is contracting if and only if there exist constants $\lambda <1$ and $C\geq 0$, and level $n$ such that for every $g\in G$ and every vertex $u$ on level $n$ 
\[ 
 \partial(g_u) \leq \lambda \partial(g) + C. 
\]
Note that, if $\partial(g)>\frac{C}{1-\lambda}$, then 
\[
 \partial(g_u) \leq \lambda \partial(g) + C < \partial(g), 
\]
that is, the sections $g_u$ are strictly shorter than $g$. Thus, the elements that possibly have no shortening in their sections are the ones that belong to the finite set 
\[ 
 \left\{g\in G\mid \partial(g)\leq \frac{C}{1-\lambda}\right\}.
\]
In particular, the nucleus is part of this finite set. The metric approach to contracting groups precedes the nucleus definition of Nekrashevych and was used by Grigorchuk in his early works on families of groups related to the first Grigorchuk group~\cite{grigorchuk:burnside,grigorchuk:gdegrees}. Nekrasevych~\cite{nekrashevych:ss-book} showed that he nucleus definition and the metric definition are equivalent in the case of finitely generated groups. 

Given a contracting group $G$ acting on a $d$-ary tree and an element $g$ in $G$, the \textbf{nucleus portrait} is a finite tree, whose interior vertices are decorated by permutations from $S_d$ and whose leaves are decorated by elements of the nucleus, describing the action of the element $g$. The portrait is constructed recursively as follows. If $g$ is an element of the nucleus the tree consist only of the root decorated by $g$. If $g$ is not an element of the nucleus, then we consider its decomposition \eqref{decompose}. The portrait of $g$ is obtained by decorating the root by the permutation $\alpha$ and by attaching the portraits of the section $g_1,\dots,g_d$ at the corresponding vertices of the first level. Since $G$ is contracting this recursive procedure must end at some point, and we obtain the portrait of the element $g$. (A concrete example in the case of the first Grigorchuk group is provided in Figure~\ref{f:portrait}.)

Let us denote by $d(g)$ the depth of the portrait of an element $g\in G$, that is, the length of the largest ray from the root to a leaf in the portrait of $g$. For each $n\in \N$, the set $\{g\in G\mid d(g)\leq n\}$ is finite, and the function $a: \N \to \N$ given by 
\[ 
 a(n) = |\{g\in G\mid d(g)\leq n\}|
\]
is called the \textbf{portrait growth function}, or portrait growth sequence, of $G$ (with respect to the nucleus). 

We now focus on regular branch groups, since their structure gives us a way to describe the portrait growth function of a contracting group in a recursive way. 

Given $G\leq\Aut T$, the elements of $G$ fixing level $n$ form a normal subgroup of $G$, called the $n$th level stabilizer and denoted $\st_G(n)$. For every $n\in\N$, we have an injective homomorphism 
\[ 
 \psi_n:\st_G(n)\longrightarrow \underbrace{\Aut T \times \dots \times\Aut T}_{d^n},
\]
sending each element $g\in\st_G(n)$ to the $d^n$-tuple consisting of the sections of $g$ at level $n$. Note that, if the group is self-similar, $\psi_n(\st_G(n))\leq G\times\overset{d^n}{\dots}\times G$. For simplicity we write $\psi=\psi_1$. A group $G\leq\Aut T$ is level transitive if it acts transitively on every level of the tree. A level transitive group $G\leq \Aut T$ is called \textbf{regular branch} if there exists a normal subgroup $K$ of finite index in $G$ such that
\[
 \psi(K\cap\st_G(1))\geq K\times\dots\times K.
\]
Observe that, since $K$ is of finite index in $G$, the group $\psi^{-1}(K\times\dots\times K)$ has finite index in $K \cap \st_G(1)$, and hence in $G$. 

We now describe a procedure yielding recursive formulas for the portrait growth of any finitely generated, contracting, regular branch group $G$, branching over its normal subgroup $K$  of index $k$. Consider a left transversal $T=\{t_1,\dots,t_{k}\}$ for $K$ in $G$ and denote by $p_n(t_i)=|\{g\in t_iK\mid d(g)\leq n\}|$ and $p_n=|\{g\in G\mid d(g)\leq n\}|$ the sizes of the sets consisting of the elements of depth at most $n$ in the coset $t_iK$ and in the whole group $G$, respectively. We have $p_n=\sum_{i=1}^{k}p_n(t_i).$  

Let $S=\{s_1,\dots,s_\ell\}$ be a left transversal for $\psi^{-1}(K\times\dots\times K)$ in $K$. For $i=1,\dots,k$ and $j = 1, \dots, \ell$ we have
\[ 
 t_is_j=(g_1,\dots,g_d)\alpha\equiv (t_{ij1},\dots,t_{ijd})\alpha\pmod{K\times\dots\times K},
\]
for some $t_{ijr}\in T$, $r=1,\dots,d$. Thus, for $n \geq 0$,  
\begin{align}\label{eq:rec}
 p_{n+1}(t_i)=\sum_{j=1}^\ell p_{n}(t_{ij1})\dots p_{n}(t_{ijd}) \\
 p_{n+1}=\sum_{i=1}^k\sum_{j=1}^\ell p_{n}(t_{ij1})\dots p_{n}(t_{ijd}).
\end{align}

The initial conditions $p_0(t_i)$, $i=1,\dots,k$, for the recursive formulas can be obtained simply by counting the members of the nucleus that come from the corresponding coset, while $p_0$ is the size of the nucleus. 

The following observation will be helpful later on. A rooted automorphism is an automorphisms whose sections, other than at the root, are trivial. 

\begin{lemma}\label{rooted-transversal}
Let $G$ be a finitely generated, contracting, regular branch group branching over the subgroup $K$ and let $T$ be a transversal of $K$ in $G$. If $a \in G$ is a rooted automorphism then, for every $t\in T$ and $n \geq 1$, we have 
\[ 
 p_n(t)=p_n(at)=p_n(ta)=p_n(t^a).
\]
\end{lemma}
\begin{proof}
Observe that for every $g\in tK$ and $u\in L_n$, $n\geq 1$, we have
\begin{align*}
(ag)_u&=a_ug_{a(u)}=g_{a(u)},\\
(ga)_u&=g_ua_{g(u)}=g_u,\\
(g^a)_u&=g_{a^{-1}(u)}.
\end{align*}
Thus, there are bijections between the sets of elements of depth $n$ in $tK$, $atK$, $taK$ and $t^aK$. 
\end{proof}


\section{Doubly exponential growth}

We begin by defining sequences of doubly exponential growth.

\begin{definition}
A sequence of positive real numbers $\{a_n\}_{n\in\N}$ grows doubly exponentially if there exist some positive constants $\alpha,\beta$ and some $\gamma,d>1$ such that
\[
 \alpha e^{\gamma{d^n}}\leq a_n\leq \beta e^{\gamma{d^n}},
\]
for every $n\in\N$.
\end{definition}

In order to show that the portrait growth sequences in our examples are doubly exponential we need the following auxiliary result.

\begin{lemma}\label{l:criterion}
Let $\left\{a_n\right\}_{n=0}^\infty$ be a sequence of positive real numbers and $d$ a constant with $d>1$. The following are equivalent.

\textup{(i)} There exist positive constants $A$ and $B$ such that, for all $n \geq 0$,
\[
 A a_n^d  \leq a_{n+1} \leq B a_n^d.
\]

\textup{(ii)} There exist positive constants $\alpha$, $\beta$, and
$\gamma$ such that, for all $n \geq 0$,
\[
 \alpha e^{\gamma d^n} \leq a_n \leq \beta e^{\gamma d^n}.
\]

Moreover, in case \textup{(i)} is satisfied, the sequence $\left\{\frac{\ln a_n}{d^n}\right\}_{n=0}^\infty$ is convergent, we may set
\[
 \gamma = \lim_{n \to \infty} \frac{\ln a_n}{d^n} 
\]
and $\alpha$ and $\beta$ can be chosen to be $e^{-M}$ and $e^M$,
respectively, where
\[
 M = \frac{1}{d-1} \max \{ |\ln A|, |\ln B| \}.
\]
The error of the approximation $\gamma \approx \gamma_n  = \frac{\ln a_n}{d^n} $ is no greater than $M/d^n$. 
\end{lemma}

\begin{proof}
(ii) implies (i). We have, for all $n$,
\[
 \frac{\alpha}{\beta^{d}} a_n^d \leq
 \frac{\alpha}{\beta^{d}}\left(\beta e^{\gamma d^n} \right)^d =
 \alpha e^{\gamma d^{n+1}} \leq
 a_{n+1} \leq \beta e^{\gamma d^{n+1}} =
 \frac{\beta}{\alpha^{d}}\left(\alpha e^{\gamma d^n} \right)^d \leq
 \frac{\beta}{\alpha^{d}} a_n^d.
\]

(i) implies (ii). For all $i$, we have $\ln A \leq \ln \frac{a_{i+1}}{a_i^d} \leq \ln
B$, and therefore
\[
\left| \ln \frac{a_{i+1}}{a_i^d} \right | \leq \max\{|\ln A|,|\ln B|\} = (d-1)M. 
\]
For $n \geq 0$, let 
\[ 
r_n = \sum_{i=n}^\infty \frac{1}{d^{i+1}} \ln \frac{a_{i+1}}{a_i^d}. 
\] 
The series $r_n$ is absolutely convergent and we have the estimate $|r_n| \leq M/d^n$, by comparison to $\sum_{i=n}^\infty \frac{1}{d^{i+1}}(d-1)M = \frac{M}{d^n}$. 

Let 
\[ 
 \gamma = \ln a_0 + r_0 =  \ln a_0 + \sum_{i=0}^\infty \frac{1}{d^{i+1}} \ln\frac{a_{i+1}}{a_i^d}. 
\]
Since $r_0$ is a convergent series, $\gamma$ is well defined. We have
\begin{align*} 
 \gamma &= \ln a_0 + r_0  = \ln a_0 + \frac{1}{d} \ln\frac{a_{1}}{a_0^d} + r_1   \\
        &= \frac{\ln a_1}{d} + r_1 = \frac{\ln a_1}{d} + \frac{1}{d^2} \ln\frac{a_{2}}{a_1^d} + r_2 \\
        &= \frac{\ln a_2}{d^2} + r_2 = \dots
\end{align*}
Thus, for all $n$, 
\[ 
 \gamma = \frac{\ln a_n}{d^n} + r_n. 
\]
Since $|r_n| \leq M/d^n$, we see that the sequence $\left\{\frac{\ln a_n}{d^n}\right\}_{n=0}^\infty$ converges to $\gamma$. Moreover, the inequalities $\gamma -M/d^n  \leq \frac{\ln a_n}{d^n} \leq \gamma + M/d^n$ yield 
\[
 e^{-M} e^{\gamma d^n} \leq a_n \leq e^M e^{\gamma d^n}. \qedhere
\]
\end{proof}

We end this section with a simple combinatorial observation that shows that an upper bound of the form required in condition (i) of the lemma always exists for regular branch groups. 

\begin{lemma}\label{upper_bound}
Let $G$ be a finitely generated, contracting, regular branch group acting on the $d$-adic tree and let $\{a_n\}_{n\geq 0}$ be the portrait growth sequence of $G$. Then
\[ 
 a_{n+1}\leq |G:\st_G(1)|a_n^d, \text{ for } n\geq 0.
\]
\end{lemma}
\begin{proof}
Since every element of depth at most $n+1$ has sections at the first level that have depth at most $n$, the number of possible decorations at level 1 and below for portraits of depth at most $n+1$ is at most $a_n^d$. On the other hand, the number of possible labels at the root is $|G:\st_G(1)|$, and we obtain the inequality.
\end{proof}


\section{Portrait growth in the first Grigorchuk group}

Denote by $\G$ the first Grigorchuk group, introduced in~\cite{grigorchuk:burnside}. In his treatise~\cite{grigorchuk:unsolved} on solved and unsolved problems centered around $\G$, Grigorchuk asked what is the growth of the sequence counting the number of portraits of given size in $\G$ (Problem~3.5). 

The first Grigorchuk group is defined as follows. 

\begin{definition}
Let $T$ be the binary tree. The first Grigorchuk group $\G$ is the group generated by the rooted automorphism $a$ permuting the two subtrees on level 1, and by $b,c,d\in\st_{\G}(1)$, where $b,c$ and $d$ are defined recursively by 
\begin{align*}
\psi(b)&=(a,c)\\
\psi(c)&=(a,d)\\
\psi(d)&=(1,b)
\end{align*} 
\end{definition}

Already in his early works in 1980s Grigorchuk observed that $\G$ is contracting with nucleus $\mathcal{N}(\G)=\{1,a,b,c,d\}$. Since $\G$ is a contracting group, its elements have well defined portraits, which are finite decorated trees. For instance, the portrait of the element $bacac$ is provided in Figure~\ref{f:portrait}.

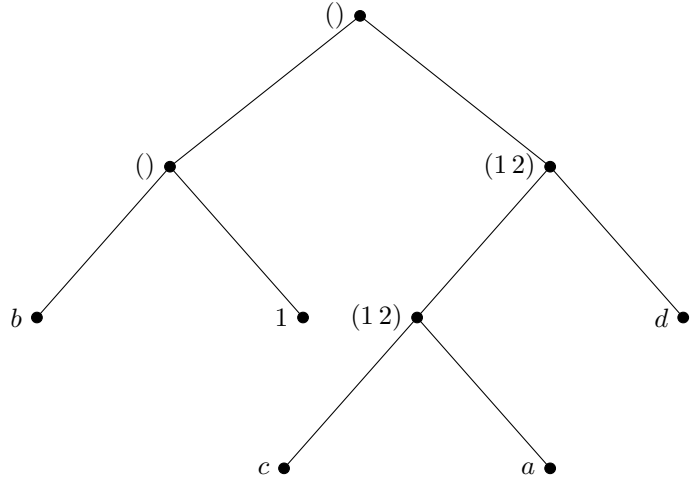
\begin{figure}[!ht]
\begin{tikzpicture}[level distance=20mm]
\tikzstyle{level 1}=[sibling distance=50mm]
\tikzstyle{level 2}=[sibling distance=35mm]
\tikzstyle{level 3}=[sibling distance=35mm]
\tikzstyle{level 4}=[sibling distance=20mm]
\coordinate (root)
child {child child}
child {child {child child} child} ; 
\draw[fill] (root) circle (2pt);
\node[left=25pt] (root){};
\draw[fill] (root-1) circle (2pt);
\node[left=2pt](root){$()$};
\node[left=2pt] at (root-1) {$()$};
\draw[fill] (root-2) circle (2pt);
\draw[fill] (root-1-1) circle (2pt);
\node[left=2pt] at (root-1-1){$b$};
\draw[fill] (root-1-2) circle (2pt);
\node[left=2pt] at (root-1-2){$1$};
\draw[fill] (root-2-1) circle (2pt);
\node[left=2pt] at (root-2-1){$(1\,2)$};
\draw[fill] (root-2-2) circle (2pt);
\node[left=2pt] at (root-2-2){$d$};
\draw[fill] (root-2-1-1) circle (2pt);
\node[left=2pt] at (root-2-1-1){$c$};
\draw[fill] (root-2-1-2) circle (2pt);
\node[left=2pt] at (root-2-1-2){$a$};
\node[left=2pt] at (root-2) {$(1\,2)$};
\node[right=15pt] at (root-2) {};
\end{tikzpicture}
\caption{portrait of the element $bacac$}
\label{f:portrait}
\end{figure}

Grigorchuk also showed that $\G$ is a regular branch group, branching over the subgroup $K= \langle [a,b] \rangle^\G $ of index $|\G:K|=16$. An accessible account can be found in Chapter VIII of~\cite{delaharpe:ggt-book}. A transversal for $K$ in $\G$ is given by 
\[
T=\{\ 1,d,ada,dada, \ a,ad,da,dad, \ b,c,aca,cada, \ ba,ac,ca,cad \ \}.
\]
and a transversal for $\psi^{-1}(K\times K)$ in $K$ is given by 
\[
S=\{1,abab,(abab)^2,baba\}.
\]

\begin{theorem}\label{t:g-rec}
The portrait growth sequence $\{a_n\}_{n=0}^\infty$ of the first Grigorchuk group $\G$ is given recursively by
\begin{align*}
 a_0 &= 5 \\
 a_n &= 2x_n + 4y_n + 2z_n + 2X_n + 4Y_n + 2Z_n, \text{ for }n \geq 1,
\end{align*}
where $x_n$, $y_n$, $z_n$, $X_n$, $Y_n$, and $Z_n$, for $n \geq 1$, satisfy the system of recursive relations
\begin{alignat*}{2}
  &x_{n+1} &&= x_n^2 + 2y_n^2 + z_n^2, \\
  &y_{n+1} &&= x_n Y_n+ Y_n z_n + X_n y_n + y_n Z_n, \\
  &z_{n+1} &&= X_n^2 + 2Y_n^2 + Z_n^2 \\
  &X_{n+1} &&= 2x_n y_n+ 2y_n z_n, \\
  &Y_{n+1} &&= x_nX_n + 2y_nY_n + z_nZ_n, \\
  &Z_{n+1} &&= 2X_n Y_n + 2 Y_n Z_n,
\end{alignat*}
with initial conditions
\begin{align*}
 x_1 &= y_1 = z_1 = Y_1 =1, \\
 X_1 &=2, \\
 Z_1 &=0.
\end{align*}
\end{theorem}

\begin{proof}
Denote by $p_n(t)$ the number of portraits of depth no greater than $n$ in $\Gamma$ that represent elements in the coset $tK$.

By Lemma~\ref{rooted-transversal}, we have
\begin{equation}\label{property_coset}
p_{n+1}(t)=p_{n+1}(at)=p_{n+1}(t^a)=p_{n+1}(ta), \text{ for } n\geq 0,~t\in T.
\end{equation} 
Thus we only need to exhibit recursive formulas for the 6 coset representatives in the set $\{1,c,dada,b,d,cada\}$. We have

\begin{align*}
\psi(1)&=(1,1) & \psi(abab)&=(ca,ac) \\ 
\psi((abab)^2)&=(dada,dada) & \psi(baba)&=(ac,ca)\\[3pt]
\psi(c)&=(a,d) & \psi(cabab)&=(aca,cad) \\
 \psi(c(abab)^2)&=(dad,ada) & \psi(cbaba)&=(c,ba)\\[3pt]
\psi(dada)&=(b,b) & \psi(dadaabab)&=(da,ad) \\
 \psi(dada(abab)^2)&=(cada,cada) & \psi(dadababa)&=(ad,da)\\[3pt]
\psi(b)&=(a,c) &\psi(babab)&=(aca,dad) \\
\psi(b(abab)^2)&=(dad,aca) & \psi(bbaba)&=(c,a)\\[3pt]
\psi(d)&=(1,b) & \psi(dabab)&=(ca, ad) \\
 \psi( d(abab)^2)&=(dada,cada) &\psi(dbaba)&=(ac,da)\\[3pt]
\psi(cada)&=(ba,d) & \psi(cadaabab)&=(ada,cad) \\
 \psi(cada(abab)^2)&=(cad,ada)&\psi(cadababa)&=(d,ba),
\end{align*}
where the sections are already written modulo $K$ by using representatives in $T$ (the spacing into 6 groups of 4 indicates how each of the 6 cosets of $K$ with representatives $1$, $c$, $dada$, $b$, $d$, and $cada$ splits into 4 cosets of $K \times K$). 

 Thus, for $n \geq 0$,
\begin{alignat}{2}\label{equations}
  p_{n+1}(1)& =p_n(1)^2 + 2p_n(ac)p_n(ca) + p_n(dada)^2,\nonumber \\
  p_{n+1}(c)& = p_n(a)p_n(d)+ p_n(dad)p_n(ada) + p_n(c)p_n(ba)+ p_n(aca)p_n(cad), \nonumber \\
  p_{n+1}(dada)& = p_n(b)^2 + 2p_n(ad)p_n(da) + p_n(cada)^2 \nonumber\\
  p_{n+1}(b) & = 2p_n(a)p_n(c)+ 2p_n(dad)p_n(aca), \\
  p_{n+1}(d) &= p_n(1)p_n(b) + p_n(ac)p_n(da) + p_n(ca)p_n(ad) + p_n(dada)p_n(cada), \nonumber\\
  p_{n+1}(cada) &= 2p_n(d)p_n(ba) + 2p_n(ada)p_n(cad),\nonumber
\end{alignat}
with initial conditions
\begin{align*}
 &p_0(1)=p_0(a)=p_0(b)=p_0(c)=p_0(d)=1, \\
 &p_0(t) = 0, \text{ for } t \in T \setminus \{1,a,b,c,d\}.
\end{align*}

Direct calculations, based on \eqref{equations}, give
\begin{align*}
 &p_1(b)=2 \\
 &p_1(cada)=0 \\
 &p_1(t) = 1, \text{ for } t \in \{1,c,d,dada\}.
\end{align*}

If we denote, for $n \geq 1$,
\begin{align*}
 x_n &= p_n(1) = p_n(a), \\
 y_n &= p_n(c) = p_n(ac) = p_n(aca) = p_n(ca), \\
 z_n &= p_n(dada) = p_n(dad), \\
 X_n &= p_n(b) = p_n(ba), \\
 Y_n &= p_n(d) = p_n(ad)=p_n(ada)=p_n(da), \\
 Z_n &= p_n(cada) = p_n(cad)
\end{align*}
we obtain, for $n \geq 1$,
\[
 a_n = 2x_n + 4y_n + 2z_n + 2X_n + 4Y_n + 2Z_n,
\]
where $x_n$, $y_n$, $z_n$, $X_n$, $Y_n$, and $Z_n$ satisfy the recursive relations and initial conditions as claimed, which follows from (\ref{equations}).
\end{proof}

\begin{theorem}
There exists a positive constant $\gamma$  such that the portrait growth sequence $\{a_n\}_{n=0}^\infty$ of the first Grigorchuk group $\G$ satisfies the inequalities
\[
\frac{1}{4} e^{\gamma{2^n}} \leq a_n \leq 4 e^{\gamma{2^n}},
\]
for all $n \geq 0$. Moreover, $\gamma\approx 0.71$.
\end{theorem}
\begin{proof}
Following Lemma~\ref{l:criterion}, we first determine positive constants $A$ an $B$ such that for each $n\in \N$ we have 
\[Aa_n^2\leq a_{n+1}\leq B a_n^2.\]
By Lemma~\ref{upper_bound} we may take $B = |\G:\st_\G(1)|=2$. 

For the other inequality, we need a constant $A$ such that
\[a_{n+1}-Aa_n^2\geq 0.\]
Using Theorem~\ref{t:g-rec} we may express, for $n \geq 1$, both $a_{n+1}$ and $a_n$ in terms of $x_n,y_n,z_n,X_n,Y_n,Z_n$ and if we set $A=\frac{1}{4}$, we obtain 
\begin{align*}
a_{n+1}-Aa_n^2=(x_n-z_n+X_n-Z_n)^2 \geq 0.
\end{align*}
Since  $M = \frac{1}{d-1} \max \{ |\ln A|, |\ln B| \} = \frac{1}{2-1} \max \{ |\ln 1/4|, |\ln 2| \} = \ln 4$, we obtain $\alpha = e^{-M} = 1/4$ and $\beta = e^M=4$. Finally, the approximation of $\gamma \approx 0.71$ can be calculated by using the recursion given by Theorem~\ref{t:g-rec} and Lemma~\ref{l:criterion}.  
\end{proof} 


\section{Portrait growth in non-symmetric GGS-groups}

The GGS-groups (named after Grigorchuk, Gupta, and Sidki) form a family of groups generalizing the Gupta-Sidki examples~\cite{gupta-sidki:burnside} (which were in turn inspired by the first Grigorchuk group~\cite{grigorchuk:burnside}). 

\begin{definition}
For a prime $p$, $p \geq 3$, and a vector $e=(e_1,\dots,e_{p-1})\in\F_p^{p-1}$, the GGS-group defined by $e$ is the group of $p$-ary automorphisms generated by the rooted automorphism $a$, which permutes the subtrees on level 1 according to the permutation $(1\dots p)$, and the automorphism $b\in\st(1)$ defined recursively by $b=(b,a^{e_1},\dots,a^{e_{p-1}})$.
\end{definition}

Set $S=S^{-1}=\{a,a^2,\dots,a^{p-1},b,\dots, b^{p-1}\}$. It is easy to see that $G$ is contracting with nucleus $\mathcal{N}(G)=S\cup\{1\}$. 

Let 
\[ 
 C(0,e)=
\begin{pmatrix}
 0 & e_1 & e_2  & \dots &e_{p-1}\\
 e_{p-1} &  0   & e_1 &\dots &e_{p-2}\\
 \vdots & \vdots  &  \vdots &\ddots &\vdots\\
 e_1 & e_2 &\dots&e_{p-1} &0
\end{pmatrix}
\]
be the circulant matrix of the vector $(0,e_1,\dots,e_{p-1})$. We say that the vector $e=(e_1,\dots,e_{p-1})$ is symmetric if $e_i=e_{p-i}$, for $i=1,\dots,p-1$ (that is, the vector is symmetric precisely when the corresponding circulant matrix is symmetric). 

\begin{theorem}
Let $G$ be a GGS-group defined by a non-symmetric vector $e\in \F_p^{p-1}$. The portrait growth sequence $\{a_n\}_{n=0}^{\infty}$ of $G$  is given by

\begin{align*}
a_0&=1+2(p-1)\\
a_n&=p(x_1+(p-1)y_1)^{p^{n-1}},
\end{align*}
where $x_1$ and $y_1$ are the number of solutions in $\F_p^{p}$ of 
\[
 ((n_0, \dots, n_{p-1})C(0,e))\odot(n_0, \dots, n_{p-1})=(0,\dots,0),
\]
with $n_0+\dots+n_{p-1}= 0$ and $n_0+\dots+n_{p-1}=1$, respectively, where $\odot$ denotes the product by coordinates. 
\end{theorem}

\begin{proof}
Fern{\'a}ndez-Alcober and Zugadi-Reizabal~\cite{alcober-zugadi:ggs} showed that a GGS-group defined by a non-symmetric vector is regular branch over $G'$, whose index in $G$ is $p^2$. A left-transversal for $G'$ in $G$ is given by

\[T=\{a^ib^j\mid i,j=0,\dots,p-1\}.\]

For each pair $(i,j)\in\{0,\dots,p-1\}^2$ denote by $p_n(i,j)$ the number of portraits of depth no greater than $n$ in the coset $a^i b^j G'$. 

We have
\[
a^ib^j\equiv a^ib^{n_0}(b^a)^{n_1}(b^{a^2})^{n_2}\dots (b^{a^{p-1}})^{n_{p-1}}\pmod{G'},
\]
where $j=n_0+\dots+n_{p-1}$ in $\F_p$.
And then,
\[a^ib^j\equiv a^i(a^{i_0}b^{n_0},\dots,a^{i_{p-1}}b^{n_{p-1}})\pmod{G'\times\dots\times G'},
\]
where $(i_0,\dots,i_{p-1})=(n_0,\dots, n_{p-1})C(0,e)$.
We obtain that,
\[
p_{n+1}(i,j)=\sum_{n_0+\dots+n_{p-1}=j}\hspace{3mm} \prod_{r=0}^{p-1}p_n(i_r,n_{r}).
\]

Observe that the decomposition of $p_{n+1}(i,j)$ does not depend on $i$, so we can set $p_{n+1}(i,j)=P_{n+1}(j)$ and we have
\begin{equation}\label{e:decomp}
p_{n+1}(j)=\sum_{n_0+\dots+n_{p-1}=j}\hspace{3mm} \prod_{r=0}^{p-1}p_n(i_r,n_{r}),
\end{equation}
and then, for $n\geq 1$, we have $a_n=p\sum_{j=0}^{p-1} p_n(j)$, where we multiply by $p$ because we have to sum for each $i=0,\dots,p-1$. 

Since the nucleus is $S \cup \{1\}$, the initial conditions are the given by 
\begin{align*}
p_0(0,0)&=p_0(i,0)=p_0(0,j)=1 \text{ for } i,j\in\{1,\dots,p-1\},\\
p_0(i,j)&=0 \text{ otherwise.}
\end{align*}
In other words, $p_0(i,j)=1$ if $ij=0$ and $p_0(i,j)=0$, otherwise. By~\eqref{e:decomp}, $p_1(0)$ is the number of solutions in $\F_p^{p}$ of 
\begin{equation}\label{e:odot}
 (n_0,\dots,n_{p-1})C(0,e)\odot(n_0,n_1,\dots,n_{p-1})=(i_0n_0,\dots,i_{p-1}n_{p-1})=(0,\dots,0),
\end{equation}
with $n_0+\dots+n_{p-1}=0$, and $p_1(j)$ is the number of solutions of the same equation, but with $n_0+\dots+n_{p-1}=j$. 

We prove by induction that $p_n(1)=p_n(j)$, for $n\geq 1$ and $j\neq 0$. Observe that, for $n=1$, if $(n_0,\dots,n_{p-1})$ is a solution of~\eqref{e:odot}, with $n_0+\dots+n_{p-1}=1$, then $(jn_0,\dots,jn_{p-1})$ is also a solution, but with $n_0+\dots+n_{p-1}=j$. Similarly, if we multiply a solution that sums up to $j$ by the multiplicative inverse $j^{-1}$ of $j$ in $\F_p$, we obtain a solution that sums up to $1$. Thus, there is a bijection between the solutions and hence $p_1(1)=p_1(j)$, for $j \neq 0$. Let us now assume that $p_n(1)=p_n(j)$, for $n\geq 1$ and $j \neq 0$, and let us prove the equality for $n+1$. By~\eqref{e:decomp} and the assumption that $n \geq 1$, we have that $p_n(i,j)=p_n(j)$ and 
\[ 
 p_{n+1}(j)=\sum_{n_0+\dots+n_{p-1}=j} \hspace{3mm} \prod_{r=0}^{p-1}p_n(n_{r}).
\]
By the inductive hypothesis we have $p_n(n_r) = p_n(j^{-1}n_r)$ (this is true regardless of whether $n_r$ is 0 or not). Thus, 
\begin{align*} 
p_{n+1}(j) &= \sum_{n_0+\dots+n_{p-1}=j} \hspace{2mm} \prod_{r=0}^{p-1}p_n(n_{r}) 
            = \sum_{n_0+\dots+n_{p-1}=j} \hspace{2mm} \prod_{r=0}^{p-1}p_n(j^{-1}n_{r}) \\
           &= \sum_{j^{-1}n_0+\dots+j^{-1}n_{p-1}=1} \hspace{2mm} \prod_{r=0}^{p-1}p_n(j^{-1}n_{r}) 
            = \sum_{n_0'+\dots+n_{p-1}'=1} \hspace{2mm} \prod_{r=0}^{p-1}p_n(n_{r}') = p_{n+1}(1).
\end{align*}
 
We now resolve the recursion~\eqref{e:decomp}. 

Denote $x_n=p_n(0)$ and $y_n=p_n(1)$, for $n \geq 1$, so that $a_n=p(x_n+(p-1)y_n)$. The fact that $p_n(n_i)= y_n$ whenever $n_i \neq 0$, together with~\eqref{e:decomp}, implies that 
\begin{align*}
x_{n+1}&=\sum_{n_0+\dots+ n_{p-1}=0}\hspace{3mm}\prod_{n_i=0}x_n\prod_{n_i\neq 0} y_n,\\
y_{n+1}&=\sum_{n_0+\dots+ n_{p-1}=1}\hspace{3mm}\prod_{n_i=0}x_n\prod_{n_i\neq 0}y_n
\end{align*}
Thus, by making all possible choices of $\ell$ coordinates, $\ell=0,\dots,p$, in $(n_0,\dots,n_{p-1})$ that are equal to 0, we obtain 
\begin{align}
x_{n+1}&=\sum_{\ell=0}^{p}x_n^{p-\ell}y_n^\ell\binom{p}{\ell}z_\ell \label{e:x}\\
y_{n+1}&=\sum_{\ell=0}^{p}x_n^{p-\ell}y_n^\ell\binom{p}{\ell}z'_\ell \label{e:y},
\end{align}
where $z_\ell$ is the number of solutions of $n_1'+\dots+n_\ell'=0$ such that none of $n_1',\dots,n_\ell'$ is 0 and $z'_\ell$ the number of solutions of $n_1'+\dots+n_\ell'=1$ such that none of $n_1',\dots,n_\ell'$ is 0. 

For $z_\ell$ and $z'_\ell$, $\ell \geq 1$, we have the relations
\begin{align*}
z_{\ell+1}&=(p-1)z'_\ell\\
z'_{\ell+1}&=z_l+(p-2)z'_\ell,
\end{align*}
with initial conditions $z_1=0$ and $z'_1=1$. The solution to this system is 
\begin{align*}
z_\ell&=\frac{1}{p}((p-1)^\ell-(-1)^{\ell-1}(p-1)),\\
z'_\ell&=\frac{1}{p}((p-1)^\ell-(-1)^\ell),
\end{align*}
which, by~\eqref{e:x} and~\eqref{e:y}, gives 
\begin{align*}
x_{n+1}&=\frac{1}{p}(x_n+(p-1)y_n)^p+\frac{p-1}{p}(x_n-y_n)^p,\\
y_{n+1}&=\frac{1}{p}(x_n+(p-1)y_n)^p-\frac{1}{p}(x_n-y_n)^p.
\end{align*}
Finally, we obtain 
\[x_{n+1}+(p-1)(y_{n+1})=(x_n+(p-1)y_n)^p,\]
and we conclude that 
\[a_n=p(x_1+(p-1)y_1)^{p^{n-1}}. \qedhere\]
\end{proof}


\section{Portrait growth in the Apollonian group}

The Apollonian group is a subgroup of the Hanoi Towers group. The Hanoi Towers group was introduced by Grigorchuk and the first author~\cite{grigorchuk-s:hanoi-crm} and the Apollonian group was introduced later by Grigorchuk, Nekrashevych and the first author~\cite{grigorchuk-n-s:oberwolfach1}. 

\begin{definition}
The Appolonian group $\Apol$ acting on the ternary tree is the group generated by the automorphisms
\begin{align*}
x&=(1,y,1)(1\,2),\\
y&=(x,1,1)(1\,3),\\
z&=(1,1,z)(2\,3).
\end{align*}
\end{definition}

Set $S=\{x,y,z,x^{-1},y^{-1},z^{-1}$. It is easy to see that $\Apol$ is contracting with nucleus $\mathcal{N}(\Apol)=S\cup\{1\}$. 

\begin{theorem}
The portrait growth sequence $\{a_n\}_{n=0}^{\infty}$ of the Apollonian group is given, for $n \geq 0$, by 
\[
a_{n}=3^{\frac{3^n-1}{2}} \cdot 7^{3^n} = \frac{1}{\sqrt{3}}\left( 7 \sqrt{3} \right)^{3^n}.
\]
\end{theorem}

\begin{proof}
Denote by $E$ the subgroup of index 2 in $\Apol$ consisting of the elements in $\Apol$ that are represented by words of even length over the alphabet $\{x^\pm,y^\pm,z^\pm\}$. A left transversal for $E$ in $\Apol$ is given by $T=\{1,x\}$. It is known~\cite{grigorchuk-s:standrews} that the Hanoi Towers group $H$ is a regular branch group over its commutator $H'$, which is of index $8$ in $H$, and that the index of $H' \times H' \times H'$ in $H'$ is $12$. The Apollonian group $\Apol$ has index $4$ in $H$ and contains the commutator $H'$. Moreover, $E=H'$, implying that $\Apol$ is a regular branch group, branching over $E$. The index of $E\times E\times E$ in $E$ is 12, and a transversal is given by 
\[ 
T' = \{1,yx,(yx)^2,x^2,y^2,z^2,x^2yx,y^3x,z^2yx,x^2(yx)^2,y^2(yx)^2,z^2(yx)^2\}.
\]

\begin{table}[!ht]
	\begin{alignat*}{6}
	&(1,1,1)      &&= 1        &&\equiv 1, \qquad\qquad\qquad 
	&&(1,y,1)(1\,2)  &&= x     &&\equiv x, \\
	& (x,y,1)(1\,3\,2)    &&= yx        &&\equiv 1, \qquad     
	&&  (y,yx,1)(2\,3)    &&= xyx       &&\equiv x, \\
	& (x,yx,y)(1\,2\,3)     &&= (yx)^2      &&\equiv 1, \qquad   
	&&(yx,yx,y)(1\,3)   &&= x(yx)^2    &&\equiv x, \\
	&(y,y,1)      &&= x^2     &&\equiv 1, \qquad  
	&&(y,y^2,1)(1\,2)   &&= x^3    &&\equiv x, \\
	&(x,1,x)   &&= y^2    &&\equiv 1, \qquad   
	&&(1,yx,x)(1\,2)   &&= xy^2    &&\equiv x, \\
	&(1,z,z)   &&= z^2    &&\equiv 1, \qquad  
	&&(z,y,z)(1\,2)   &&= xz^2    &&\equiv x,\\
	&(yx,y^2,1)(1\,3\,2)       &&= x^2yx        &&\equiv 1, \qquad\qquad\qquad 
	&&(y^2,y^2x,1)(2\,3)     &&= x^3yx     &&\equiv x, \\
	&(x^2,y,x)(1\,3\,2)     &&= y^3x        &&\equiv 1, \qquad     
	&&(y,yx^2,x)(2,3)     &&= xy^3x       &&\equiv x, \\
	&(x,zy,z)(1\,3\,2)      &&= z^2yx     &&\equiv 1, \qquad   
	&&(zy,yx,z)(2\,3)   &&= xz^2yx    &&\equiv x, \\
	&(yx,y^2x,y)(1\,2\,3)      &&= x^2(yx)^2     &&\equiv 1, \qquad  
	&&(y^2x,y^2x,y)(1\,3)   &&= x^3(yx)^2    &&\equiv x, \\
	&(x^2,yx,xy)(1\,2\,3)   &&= y^2(yx)^2   &&\equiv 1, \qquad   
	&&(yx,yx^2,xy) (1\,3)  &&= xy^2(yx)^2    &&\equiv x, \\
	&(x,zyx,zy)(1\,2\,3)   &&= z^2(yx)^2    &&\equiv 1, \qquad  
	&&(zyx,yx,zy)(1\,3)  &&= xz^2(yx)^2    &&\equiv x,\\
	\end{alignat*}
	\caption{The cosets of $E \times E\times E$ decomposing the cosets of $E$}
	\label{tb:Edecomposition}
\end{table}

Denote by $X_n$ and $Y_n$ the number of portraits of depth at most $n$ in the cosets $1E$ and $xE$ respectively. The coset decomposition provided in Table~\ref{tb:Edecomposition} implies that 
\begin{align*}
X_{n+1}&=3X_n^3+9X_nY_n^2,\\
Y_{n+1}&=3Y_n^3+9X_n^2Y_n.
\end{align*}
which then yields
\[
a_{n+1}=X_{n+1}+Y_{n+1}=3(X_n+Y_n)^3=3a_n^3.
\]
Taking into account the initial condition $a_0=7$, we obtain $a_{n}=3^{\frac{3^n-1}{2}}7^{3^n}$ by induction.
\end{proof}


\end{document}